\documentclass{article}

\usepackage{amsmath,amsthm,amsfonts,amssymb}
\usepackage{graphicx}
\usepackage{tikz-cd}

\newtheorem{lm}{Lemma}
\newtheorem{prop}[lm]{Proposition}
\newtheorem{theorem}[lm]{Theorem}

\newtheorem{cy}[lm]{Corollary}

\newtheorem{df}{Definition}
\theoremstyle{definition}

\newtheorem{rk}[lm]{Remark}

\usepackage{color,cancel}

\newcommand{\beq}{\begin{equation}}
\newcommand{\eeq}{\end{equation}}
\newcommand{\be}{\begin{enumerate}}
\newcommand{\ee}{\end{enumerate}}
\newcommand{\bp}{\begin{proof}}
\newcommand{\ep}{\end{proof}}
\newcommand{\bi}{\begin{itemize}}
\newcommand{\ei}{\end{itemize}}
\newcommand{\bea}{\begin{eqnarray*}}
\newcommand{\eea}{\end{eqnarray*}}
\newcommand{\bml}{\begin{multline*}}
\newcommand{\eml}{\end{multline*}}

\newcommand{\prn}[1]{\left( #1 \right)}
\newcommand{\bkt}[1]{\left[ #1 \right]}
\newcommand{\set}[1]{\left\{ #1 \right\}}
\newcommand{\abs}[1]{\left| #1 \right|}
\newcommand{\norm}[1]{\left|\left| #1 \right|\right|}

\newcommand{\gen}[1]{\langle #1 \rangle}
\newcommand{\Z}{\mathbb{Z}}
\newcommand{\C}{\mathbb{C}}
\newcommand{\Q}{\mathbb{Q}}

\edef\storedcatcodeat{\the\catcode`\@} \catcode`\@=11

\catcode`\@=\storedcatcodeat\relax

\begin{document}

\title{Quantifying separability in RAAGs via representations}

\author{Olga Kharlampovich and Alina Vdovina}

\maketitle
\begin{abstract} We answer the question in \cite{L} and  prove the following statement. Let $L$ be a  right-angled Artin group (abbreviated as RAAG), $H$  a word quasiconvex subgroup of $L$, then there is a finite dimensional representation of $L$ that separates the subgroup $H$ in the induced Zariski topology. As a
corollary, we establish a polynomial upper bound on the size of the quotients used to
separate $H$ in $L$. This implies the same statement for a virtually special group $L$ and, in particular,  a fundamental groups of a hyperbolic 3-manifold.
   
\end{abstract}
\section{Introduction}

A subgroup $H<G$ is {\em separabl}e if  for any $g\in G-H$ there exist a homomorphism $\phi:G\rightarrow K$, where $K$ is finite and $\phi (g)\not\in \phi (H).$ 
Alternatively, $H=\cap _{H\leq L\leq G, [G:L]<\infty } L$.
Residual finiteness means that the trivial subgroup $1<G$ is separable. It was shown in \cite[Theorem F]{Hag} that every word quasiconvex subgroup of a finitely generated right-angled Artin
group (RAAG) is a virtual retract, and hence is separable.  If $B$ is a virtually special  compact cube complex such that $\pi _1(B)$ is
word-hyperbolic,  then every quasiconvex subgroup
of $\pi _1(B)$ is separable \cite{HW}. For both these cases we will quantify separability.
 Namely, we answer the question in \cite{L} and  prove the following statement. Let $L$ be a RAAG,   if $H$ is a cubically convex-cocompact subgroup of $L$, then there is a finite dimensional representation of $L$ that separates the subgroup $H$ in the induced Zariski topology. As a
corollary, we establish  polynomial upper bound on the size of the quotients used to
separate $H$ in $L$. This implies the same statement for a virtually special group $L$ and, in particular,  a fundamental groups of hyperbolic 3-manifold.  


\begin{theorem}\label{th1} Let $H\leq L$ be one of the following pairs of groups:

(1) $L$ a RAAG, $H$ a word quasiconvex subgroup;

(2) $L$ a virtually special group, $H$ a word quasiconvex  subgroup;
 
 (3)  $L$ a hyperbolic virtually special group,  $H$ a quasiconvex subgroup.
 
 Then there is a  faithful representation $\rho _H:L\rightarrow GL(V)$ such that $\overline {\rho _H(H)} \cap\rho_H(L)=\rho _H(H)$, where $\overline {\rho _H(H)} $ is the Zariski closure of  $\rho _H(H).$ 
\end{theorem}

\begin{cy} \label{co1} Let $L$  and $H$ be as in the theorems above.  Then there exists a constant $N>0$ such that for each $g\in L- H,$ there exist a finite group $Q$ and a homomorphism $\varphi:L\longrightarrow Q$ such that $\varphi\prn{g}\notin\varphi\prn{H}$ and $\abs{Q}\leq\norm{g}_S^N.$ If $K=H\ker\varphi,$ then $K$ is a finite-index subgroup of $L$ whose index is at most $\abs{Q}\leq\norm{g}_S^N$ with $H\leq K$ and $g\notin K.$ Moreover, the index of the normal core of the subgroup $K$ is bounded above by $\abs{Q}$.
\end{cy}

The groups covered by this corollary include fundamental groups of hyperbolic
3-manifolds, $C'(1/6)$ small cancellation groups and, therefore, random groups for density less than $1/12$.

Theorem \ref{th1} and Corollary \ref{co1} generalize results for free groups, surface groups from \cite{L} and for limit groups \cite{BK}.  We use \cite{L} to deduce Corollary \ref{co1} from the theorems.
 The constant $N$ in Corollary \ref{co1} explicitly depends on the subgroup $H$ and the dimension of $V$ in
Theorem \ref{th1}. It is known that there exists a finite index subgroup separating an element $g$ from $H$ of order linear in $||g||$ \cite{HP}, but the  upper bound for the index
of the normal core is factorial in the index of the subgroup. It is for this reason that
we include the statement about the normal core of $K$ at the end of the corollary.

Recently, several effective separability results have been established; see \cite{BR1}-\cite{HP},  \cite{L}-\cite{K3},  \cite{P1}-\cite{R}, \cite{S}.
Most relevant here are  papers \cite{L}, \cite{HP}.  
The methods used in \cite{HP} give linear bounds in terms of the word length of $|g|$ on the
index of the subgroup used in the separation but do not  produce polynomial
bounds for the normal core of that finite index subgroup.

\section{Preliminaries}
We now recall some terminology. We mostly follow \cite{Hag}, \cite{Wise}, \cite{HP}.  A subset $S$ of a geodesic metric space
$X$ is $K$-{\em quasiconvex} if for every geodesic $\gamma$ in $X$ whose endpoints lie in $S,$
the $K$-neighborhood of $S$ contains $\gamma$. 

We say that $S$ is {\em convex} if it is
$0$-quasiconvex.  A subcomplex $Y$ of a CAT(0) cube complex $X$ is convex provided $Y$ is connected
and for each vertex $v$ of $Y$ the link of $Y$ at $v$ is a full subcomplex of the link of $X$ at $v.$ The {\em combinatorial
convex hull} of a subcomplex $Y \subset X$ is the intersection of all convex subcomplexes
containing $Y$ .

A group $H$ acting on a geodesic metric space $X$ is {\em quasiconvex} if the
orbit $Hx$ is a $K$-quasiconvex subspace of $X$ for some $K > 0$ and some
$x\in X.$ If $H$ preserves a convex closed subset $C$ and acts cocompactly on $C$ we
say that $H$ is {\em convex cocompact}. 

A right-angled Artin group (often abbreviated as RAAG) is a group which has a presentation whose only relations are commutators between generators. One can consider generators as vertices of a graph $\Gamma$ that has an edge between two vertices if and only if the corresponding generators commute. Then the group is denoted by $A(\Gamma )$.  This group is the fundamental group of the associated  Salvetti complex $S(\Gamma )$ . The universal cover  of $S(\Gamma )$ is a CAT(0) cube complex on which $A(\Gamma )$ acts properly and cocompactly by isometries.  $A(\Gamma )$ is also the fundamental group of the 2-skeleton of $S(\Gamma )$. 

We will use the previous notions in the following context:  $X$ is the set of
vertices of a cube complex equipped with the combinatorial distance. Here
a geodesic is the sequence of vertices of a combinatorial geodesic of the
1-skeleton.  
We say that $H$ is {\em combinatorially quasiconvex} or {\em word quasiconvex}.  If $X$ is the set of vertices of a Salvetti complex  $S(\Gamma)$, then geodesics correspond to geodesic words in the standard generators of the RAAG.

 Convex cocompact subgroups of RAAGs are virtual retracts. 
If a  subgroup  $H$ of a RAAG is word quasi-convex  then it is convex cocompact  (with the closed subset being the  convex hull of a quasiconvex orbit)\ \cite[Theorem H]{Hag}. 
For hyperbolic groups the notion of a quasi-convex subgroup do not depend on a generating set (see, for example, \cite{Short}).

A {\em local isometry} $\phi :Y\rightarrow X$ of cube complexes is a locally
injective combinatorial map with the property that, if $e_1,\ldots ,e_n$ are edges  of $Y$ all incident to a $0$-cube (vertex)
$y$, and the (necessarily distinct) edges  $\phi (e_1),\ldots ,\phi (e_n)$ all lie in a common $n$-cube $c$ (containing $\phi(y)$),
then $e_1,\ldots ,e_n$ span an $n$-cube $c_0$ in $Y$ with $\phi (c_0)=c$. If $\phi :Y\rightarrow X$ is a local isometry and $X$ is
nonpositively-curved, then $Y$ is as well and the map is $\pi _1$-injective \cite[Lemma 2.12]{Hag}.  Moreover, $\phi$ lifts to an embedding $\tilde\phi :\tilde Y\rightarrow \tilde X$
of universal
covers, and $\tilde Y$ is convex in $\tilde X$ \cite[Lemma 3.2]{Wise}.


We will need the following fact used in the proof of Corollary B in  \cite{HP}.
 \begin{lm}  \label{leH} Let $L=A(\Gamma )$ be a RAAG, $S(\Gamma )$ its Salvetti complex, $H$ is a word quasiconvex subgroup of $L$.   Then $H=\pi _1(Y)$,  where $Y$ is a compact connected cube complex, based at a vertex
$w$, with a based local isometry $f: Y\rightarrow S(\Gamma )$.  \end{lm}
\begin{proof} Since $H$ is word quasiconvex,  there exists a quasiconvex $H$-orbit.  Therefore,  as in \cite[Theorem H]{Hag}, the convex hull $\tilde Y$ in the universal cover $\tilde S(\Gamma)$ of this $H$-orbit has $H$ acting on $\tilde Y$ cocompactly, and one takes $Y=H\backslash {\tilde Y}$, with $f$ induced by the inclusion $\tilde Y\rightarrow \tilde S(\Gamma )$.
\end{proof}

A nice thing about special cube complexes is
the fact that they behave in several important ways like graphs. One of the
features of special cube complexes is the ability to extend compact local isometries to covers,
generalizing the fact that finite immersions of graphs extend to covering maps. This
procedure, introduced in \cite{HW} and \cite[Definition 3.2] {HW12}  and outlined for the case that we need in  \cite[Theorem 2.6]{BR2}  is called {\em canonical completion}. We need this construction to analyse fundamental groups of $Y$ and $C(Y)$ so we describe the procedure here.

 Let $\Gamma$ be a finite graph, $S(\Gamma )$ Salvetti complex, $Y$ a compact cube complex that has a local isometric embedding in $S(\Gamma )$ with respect to combinatorial metric (so $\pi _1(Y)\leq A(\Gamma )).$  Since  $Y$  has a local isometric embedding in $S(\Gamma )$, each edge in $Y$ is naturally equipped with a direction and
labelling is induced by a generator of $\pi _1(S(\Gamma ))$. 

 For each label $x$ we consider maximal connected paths with all the edges labelled by $x$. If such a path $p$ is a cycle, we leave it alone.
If the initial vertex $v$ of the path $p$ has valency 1 with respect to $x$, then the final vertex $u$ has valency 1 as well because of $Y$ being locally isometrically
embedded in $S(\Gamma )$. In this case we add an edge labelled by $x$ with the initial vertex $u$ and the final vertex $v$. If a vertex $w$ does not have edges with label
$x$ adjacent to it, we add a loop with the label $x$ to the vertex $w$. Let $\hat Y$ be a union of $Y$ and all these new edges.   The crucial result is 
\begin{lm} \label{2.7} \cite[Lemma 2.7]{BR2} For each 2-cube $c$ of $S(\Gamma )$  the boundary path of $c$ lifts to a closed path in $\hat Y$.
\end{lm}

Hence we can add  cubes to $\hat Y$ everywhere where there is a boundary of a cube.  See the example in Fig 1 (we basically borrowed it from N. Lazarovich's talk in Montreal). Therefore $Y$ can be embedded in $C(Y)$ that is a finite cover of $S(\Gamma )$.  We now state the main properties of the canonical completion.

\begin{prop} \cite[Theorem 2.6]{BR2}\label{leHa} For each standard generator $x$ of  $\pi _1(S(\Gamma ))$, let $x$ denote the corresponding (oriented) edge of $S(\Gamma )$. Then for each vertex $w$ of $C(Y)$, we have that $w\in Y\subset C(Y)$, and there is an embedded cycle $\sigma$ in $C(Y)$ that starts and ends at $w$ such that $\sigma$ maps surjectively to $x$ under the covering map, and $\sigma$ contains at most one edge $\hat x$ labeled by $x$ not belonging to $Y$.

Moreover, the one-skeleton of $C(Y)$ consists on the one-skeleton of $Y$, together with the various edges $\hat x$ described above.
\end{prop}

\begin{figure}[ht!]
\centering
\includegraphics[width=1\textwidth]{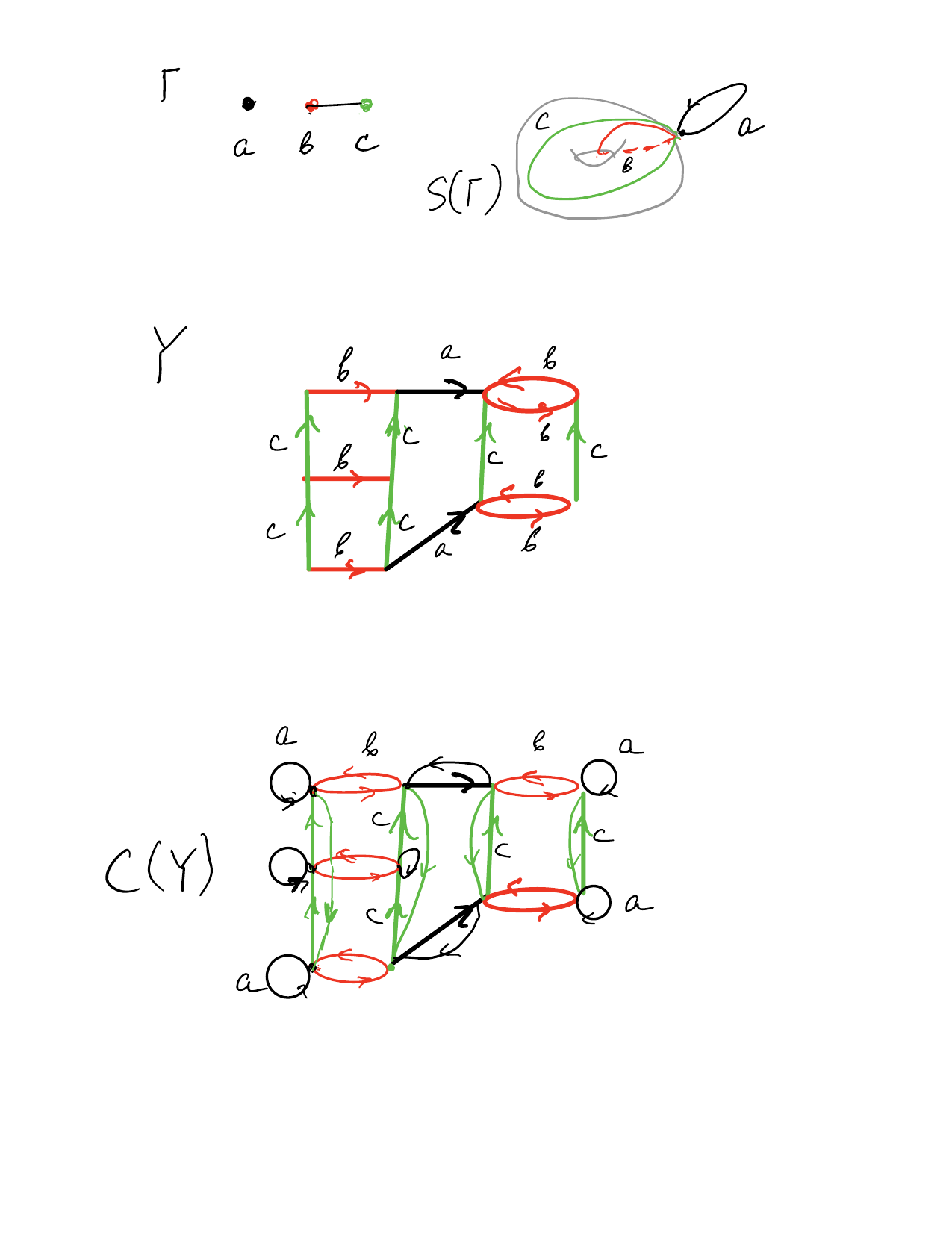}
\caption{There is a local isometry $f: Y\rightarrow S(\Gamma )$ and $C(Y)$ is the canonical completion of $Y$. Only 1-skeletons are shown.}
\label{ext1}
\end{figure}

\newpage
\section{Fundamental group of the canonical completion}

In this section we will understand the structure of the fundamental group  of the canonical completion.
\begin{prop}\label{ext}  Let $L$ be a RAAG, $H$ a word quasiconvex subgroup. Then   there exists a finite index subgroup $K$ obtained from $H$  by adding some conjugates of powers of standard generators of $L$. 

Algebraically, $K$ is obtained from $H$ by a series of HNN-extensions such that associated subgroups are the same and with identical isomorphism (free products with infinite cyclic group occur as a special case, when the associated subgroup is trivial). 

\end{prop}
\begin{proof} Let $L=A(\Gamma )$ be a RAAG, $S(\Gamma )$ its Salvetti complex, then by Lemma \ref{leH},   $H=\pi _1(Y)$,  where $Y$ is a compact connected cube complex, based at a vertex
$w$, with a based local isometry $f: Y\rightarrow S(\Gamma )$.
We often use the following observation. 
\begin{rk} \label{rk1}  Let $x,y$ be standard generators of $L$ or their inverses such that $[x,y]=1$ in $L$. 

1. Suppose there are edges $(v_1,v_2)$ labelled by $x$ and $(v_1, w_1)$ labelled by $y$ in $Y$.  Since $f: Y\rightarrow S(\Gamma )$ is a local isometry, these two edges must be on the boundary of a square. Therefore  there are edges $(w_1,w_2)$ with label $x$ and $(v_2,w_2)$ with label $y$ in $Y$.

2.  By induction we obtain the following. Suppose there is a path  $(v_1,\ldots v_k)$ with edges labelled by $x$ ($x$-path) and  an edge $(v_1, w_1)$  labelled by $y$ in $Y$.  Then there is an $x$-path  
$(w_1,\ldots ,w_k)$ , edges $(v_i,w_i)$ labelled by $y$ ($i=2,\ldots ,k$) in $Y$ and for each $i=1,\ldots ,k-1$, vertices $v_i,v_{i+1},w_i,w_{i+1}$ are corners of a square.

\end{rk}

The first statement of the proposition follows from Lemma \ref{leHa}.

Now we will prove the second statement  by induction on the number of steps needed to complete the one-skeleton of $Y$.  We will only work with 2-skeletons since they determine fundamental groups. For a label $x$ we  say that a vertex $v$ is {\em complete in $x$} if there is a loop labelled by $x^k$ (for some $k$)  starting at $v$. 

We take an incomplete in $x$ vertex $v$ of $Y$.  A step is the following transformation.

 Suppose that  $v$ has valency 1 in $x$ and  $v$ has an outgoing edge labelled by $x$. Let $v_k$
be the endpoint of the maximal $x$-path $(v=v_1,\ldots ,v_k)$  from $v$. Then we connect $v_k$ with $v$ by an edge labelled by $x$.  Applying the second statement of Remark \ref{rk1} inductively, we see that each vertex $w$   that is connected to $v$ by a path with label $q$ that commutes with $x$, also has an $x$-path $(w=w_1,\ldots ,w_k)$  with each $w_i$ connected to $v_i$ by a path with label $q$. Moreover, this path must be a maximal $x$-path  that starts at $w_1$ (otherwise we could apply to it  the remark and get a contradiction with maximality of $(v_1,\ldots ,v_k)$).  We connect each such $w_k$ with $w_1$ by an edge labelled by $x$. 

If $v$ has valency 0 in $x$, we   add a loop labelled by $x$ in $v$ and to all vertices connected to $v$ by  paths with labels that commute with $x$. 

Suppose, finally, that $v$ has valency 2 in $x$. We complete the path labelled by a power of $x$  passing through $v$. We also complete the corresponding $x$-labelled path  in all vertices connected to $v$ by a path with label that commutes with $x$. It follows from Remark \ref{rk1} that we can do this.

By Lemma \ref{2.7}, if in the obtained extension of $Y$ (denoted $Y_1$) we have a lift of the boundary of a square, then this lift is a closed path. To complete the step, we fill in newly obtained lifts of boundaries of squares by the squares. Denote the obtained complex by $Y_2$.  Both $Y, Y_2$ have a local isometry into $S(\Gamma )$. 

 In all the  cases we added an extra generator that is a conjugate to $x^k$, where $k$ was the length of the completed path. We can assume that  the base point is at  $v$,  and denote $t=x^k$.
 
  Let $\Gamma _0$ be the subgraph of $\Gamma$ obtained by removing the vertex $x$ and all the edges incident to it.  Then $L$ can be considered as an HNN-extension of $A(\Gamma _0)$ with the  stable letter $x$ and edge group $Q$, where $Q$ is the subgroup of $L$ generated by all the standard generators (except $x$) that commute with $x$ (with the identity isomorphism $Q\rightarrow Q$).  Reduced  forms of elements in HNN-extensions are defined, for example, in  \cite[Section 4.2]{LS}.   Let $h_0,\ldots ,h_n\in A(\Gamma _0)$. An element $$h_0x^{k_1}h_1x^{k_2}\ldots x^{k_n}h_n\in L$$  is in {\em reduced form}   if either $n=0$ or $n\geq 1$,  $k_1,\ldots ,k_n\neq 0$,  $h_2,\ldots ,h_{n-1} \not\in Q$ (by other words, $h_1,\ldots ,h_{n-1}$ do not commute with $x$). 
 By Britton's Lemma \cite[Section 4.2]{LS}, for $n\neq 0$ an element in reduced form is nontrivial in $L$.
 \begin{rk}\label{rk2}  1. If a relation  $y_1y_2=y_2y_1,$ where $y_1,y_2$ are standard generators of $L$, is applied to a label $g\in\pi _1(Y,v)$ of a loop at $v$ in the 1-skeleton of $Y$, then there is another loop at $v$ in the 1-skeleton of $Y$ labelled by the result.  
 
2. Since $g$ can be brought to a reduced form by  a sequence of applications of commutativity relations and cancellations, there is a loop at $v$ in the 1-skeleton of $Y$ labelled by this reduced form.
 \end{rk}
 Indeed, let $g\in\pi _1(Y,v)$ labels a loop $\sigma$.   If $\sigma $  is a composition of paths  $\bar p_1\bar  p_2\bar p_3$, and $\bar p_2$ is labelled by $y_1y_2$, then  by  Remark \ref{rk1}, the path $\bar p_2$ can be replaced by a path $\bar p_4$ labelled by $y_2y_1$ that is  also in the 1-skeleton of $Y$, as well as the loop $\bar p_1\bar p_4\bar p_3$.

Now we claim that $\pi _1(Y_2,v)$ is an HNN-extension of $\pi _1(Y,v)$ with stable letter $t$ and the edge group being the intersection  $\pi _1(Y, v)\cap Q$.  Obviously, in $\pi _1(Y_2,v)$  all the elements from $\pi _1(Y,v)\cap Q$ commute with $t$. We have to show that there are no other relations involving $t=x^k$.

Suppose there is a cyclically reduced relation in $\pi _1(Y_2)$ \begin{equation} \label{rel}x^{n_0k}g_1x^{n_1 k}g_2\ldots g_m=1,\end{equation} where $g_1,\ldots ,g_m\in\pi _1(Y,v)$ all do not commute with $x$ in  $L$, $n_0,\ldots ,n_{m-1}\neq 0$.   
Since a reduced form of the trivial element in $L$ does not contain $x$, there should be possible to make  a reduction in $L$ such that all powers of $x$ in (\ref{rel}) including those that occur in $g_1,\ldots ,g_m$,  cancel with each other. If $x^{n_jk}$  completely cancels with powers of $x$ in  $g_j$ and $g_{j+1}$ for some $j$,  then  they can be represented in reduced form as labels of loops at $v$ in the 1-skeleton of $Y$ as
$$g_j=q_0x^{m_1}\ldots x^{m_t}q_t, \  g_{j+1}=p_0x^{s_1}\ldots x^{s_r}p_r,$$
where $q_0,\ldots ,q_t,p_0,\ldots , p_r$ do not contain $x$, $q_t,p_0$ commute with $x$, $m_t+s_1=-n_jk$ and $q_{t-1},p_1$ do not commute with $x$ (because $g_j,g_{j+1}$ do not commute with $x$).   We can rewrite  $g_j,g_{j+1}$ in reduced form as $$g_j=q_0x^{m_1}\ldots q_tx^{m_t}, \  g_{j+1}=x^{s_1}p_0\ldots x^{s_r}p_r.$$  By Remark \ref{rk2} they also must correspond to the labels of loops in the 1-skeleton of $Y$ starting at the vertex $v$. In the 1-skeleton of $Y$ there is no loop labelled by a power of $x$ at $v$ (recall, that we added one edge labelled by $x$  and  obtained the loop labelled by $x^k$ in the 1-skeleton of  $Y_2$).  Therefore in $Y$ we cannot have both: a path labelled by $x^{m_t}$ ending at $v$ and a path labelled by $x^{s_1}$ beginning at $v$ such that $m_t+s_1=-n_jk$ (because $|m_t+s_1|<k$).  Therefore $x^{n_jk}$ cannot cancel with powers of $x$ in $g_j$ and $g_{j+1}$.  Therefore either $g_j$ or $g_{j+1}$  must commute with $x$ in $L$, and there are no relations in the form (\ref{rel}) in $\pi_1(Y_2,v).$ 
 This implies that all relations in $\pi _1(Y_2,v)$ involving $x^k$ follow from relations  $[x^k, \bar g]=1$, where $\bar g\in \pi _1(Y,v)$. 

Suppose $[\bar g,x^k]=1$, $\bar g\in\pi _1(Y,v)$, then $[\bar g,x]=1$ in $L$.
Then $\bar g=x^{k_1}g$, where $g\in L,$ $[g,x]=1$ and $g$ does not contain $x$ when written in the standard generators of $L$. If $k_1=0$, then $g\in \pi _1(Y,v)\cap Q$ and we are done. Suppose $k_1\neq 0$.  Then, since $x^k\in \pi _1(Y_2,v)$,   $g^{k}\in \pi _1(Y_2,v)$. Therefore $g^{k}\in \pi _1(Y,v)$ 
But then $x^{k_1k}\in\pi _1(Y,v),$ contradiction with the assumption that $v$ was $x$-incomplete. Therefore $k_1=0$ and $g\in \pi _1(Y,v)\cap Q$. Relation $[x^k,g]=1$ is an HNN-extension relation.

Therefore, 
 all relations of  $\pi _1(Y_2,v)$ containing $x^k$ are consequences of relations $[x^k,g]=1$ for   elements $g\in\pi _1(Y,v)\cap Q$. This proves the claim. Since the canonical completion is constructed by a sequence of such steps, the proposition is proved.

\end{proof}

\section{Representations}
Given an algebraically closed field $K$, a finite dimensional $K$-vector space $V$, a finitely
generated group $G$, and a homomorphism $\rho : G\rightarrow GL(V)$, we have the subspace
topology on $\rho (G)$ induced by the Zariski topology on $GL(V)$. The
pullback of this topology to $G$ under $\rho$ is called the Zariski topology associated to $\rho$.
 \begin{df}\cite{L}
   Let $G$ be a finitely generated group and $H$ a finitely generated  subgroup of $G.$ For a complex affine algebraic group $\bf{G}$ and any representation $\rho_0\in Hom(G,\bf{G}),$ we have the closed affine subvariety 
    $$R_{\rho_0,H}(G,{\bf{G}})=\set{\rho\in Hom(G,{\bf G}): \rho_0(h)=\rho(h)\text{ for all } h\in H}$$
    The representation $\rho_0$ is said to \textit{strongly distinguish} $H$ in $G$ if there exist representations $\rho,\rho'\in R_{\rho_0,H}(G,\bf{G})$ such that $\rho(g)\neq\rho'(g)$ for all $g\in G-H.$
\end{df}

\begin{lm}\label{3.1}\cite[Lemma 3.1]{L} Let $G$ be a finitely generated group, ${\bf G}$ a complex algebraic group, and $H$ a finitely generated subgroup of $G.$ If $H$ is strongly distinguished  by a representation $\rho\in Hom(G, {\bf G})$, then there exists a representation $\varrho: G\longrightarrow {\bf G}\times {\bf G}$ such that $\varrho(G)\cap\overline{\varrho(H)}=\varrho(H),$ where $\overline{\varrho(H)}$ is the Zariski closure of $\varrho(H)$ in ${\bf G}\times {\bf G}.$
\end{lm}

\noindent

\begin{prop}  Let $L$ be a RAAG and $H$ a word quasiconvex  finitely generated subgroup. There exist a finite index subgroup  $K\leq L$, a complex affine algebraic group $\bf G$ and  a faithful representation $\rho :K\rightarrow \bf G$ that  strongly distinguishes $H$ in $K$.
 \end{prop} 
 \begin{proof} 
 
Let $K$ be a finite index subgroup from Proposition \ref{ext}.
RAAGs are linear, so $K$ is faithfully represented in $GL(k, {\mathbb C})$ for some $k$ (for many classes of RAAGs, $k\leq 4$).  Let $\rho_0$ be a faithful representation of $K$ in $GL(k, {\mathbb C})$. 

Since $K$ is obtained from $H$ as in Proposition \ref{ext}, we can write
 $$H=K_0<\ldots K_i<\ldots K_n=K,$$
 where $K_i=\langle K_{i-1},t_i|[H_{i-1},t_i]=1\rangle,$  where $H_{i}<K_{i}$.

Define representations $\rho _j:K\rightarrow GL(k, {\mathbb C})$ for $j=1,\ldots ,n$ as follows:

\vspace{0.1cm}
$\rho _j(g)=\rho _0(g)$ for $g\in K_{j-1}$, $\rho _j(t_j)=\rho_0(t_j)^2$ and $\rho _j(t_i)=1$ for $i>j.$

Notice that $\rho _j$ is faithful on $K_j$ because the subgroup generated by $K_{j-1}$ and $t_j^2$ in $K_j$ is isomorphic to $K_j$.  
For $g\in K_j-K_{j-1}$  consider an element $g_1$ obtained from the normal form of $g$ in the HNN-extension $K_j$ by replacing each $t_j$ by $t_j^2$. It follows from the normal form of elements in $K_j$  that  $g\neq g_1$.   Since $\rho _0$ is faithful, we have  $\rho _0(g)\neq \rho _0(g_1)=\rho _j(g)$.

Now define $\rho, \rho ': K\rightarrow GL(k, {\mathbb C})^n$ as $\rho (g)=\rho _0(g)\times\ldots \times\rho _0(g)$ and  $\rho'(g)=\rho _1(g)\times\ldots\times\rho _n(g).$ Take ${\bf G}=GL(nk, {\mathbb C})$ since  $GL(k, {\mathbb C})^n\leq GL(nk, {\mathbb C})$.

 Then $H$ is strongly distinguished in K
 by the representation $\rho$, because  $\rho$ and $\rho '$ are the same on $H$ and $\rho ' (g)\neq\rho (g),$ for any $g\in K-H$. 

 \end{proof}

\noindent Let us prove Theorem \ref{th1}.   The proof of \cite[Theorem 1.1]{L}
shows that it is sufficient to have a representation of $K$ that strongly distinguishes $H$. Indeed, like in \cite[Corollary 3.4]{L}, we can construct a representation
$\Phi : K\rightarrow GL(k,{\mathbb C})\times GL(k,{\mathbb C})$ for some $k$ such that $\Phi (g)\in Diag(GL(k,{\mathbb C}))$ if and only if $g\in H$. Setting $d_H=[L:K]$, we have the induced representation $${Ind_K}^G(\Phi ):L\rightarrow 
GL(kd_H,{\mathbb C})\times GL(kd_H,{\mathbb C}).$$  
Recall, that when $\Phi$ is represented by the action on the vector space $V$ and $L=\cup _{i=0} ^t g_iK$, then the induced representation 
acts on the disjoint union $\sqcup_{i=0} ^t g_iV$ as follows
$$g\Sigma g_iv_i=\Sigma g_{j(i)} \Phi (k_i)v_i,$$ where
$gg_i=g_{j(i)}k_i, $ for $k _i\in K.$ Taking $\rho ={Ind_K}^G(\Phi )$, it follows from the construction of $\rho $ and definition of induction that $\rho (g)\in\overline {( \rho (H))}$ if and only if $g\in H$. If we set $\rho=\rho_H,$  then  Theorem \ref{th1}(1) is proved. \\

Since a special group is a subgroup of a RAAG and we can extend a representation from a finite index subgroup to the whole group, (1)  implies  (2).
All definitions of quasiconvexity coincide in hyperbolic groups and (3) follows from (2).

\section{ Proof of Corollary \ref{co1}}

Given a complex algebraic group $\textbf{G} < GL(n,\C)$, there exist polynomials  $P_1,\ldots,P_r\in\C\bkt{X_{i,j}}$ such that
 $$\textbf{G}=\textbf{G}\prn{\C}= V\prn{P_1,\ldots,P_r}=\set{X\in\C^{n^2}\mid P_k(X)=0, k=1,\ldots, r}$$
We refer to the polynomials $P_1,\ldots,P_r$ as \textit{defining polynomials}  for $\textbf{G}$. We will say that \textbf{G} is $K$-defined for a subfield $K\subset\C$ if there exists defining polynomials $P_1,\ldots,P_r\in K\bkt{X_{i,j}}$ for $\textbf{G}$.
For a complex affine algebraic subgroup $\textbf{H} < \textbf{G} < GL(n,\C)$, we will pick the defining polynomials for $\textbf{H}$ to contain a defining set for $\textbf{G}$ as a subset. Specifically, we have polynomials
$P_1,...,P_{r_{\textbf{G}}},P_{r_{\textbf{G}}+1},...,P_{r_{\textbf{H}}}$  such that 
 \begin{equation}\label{eq:1}
     \textbf{G}= V\prn{P_1,\ldots,P_{r_{\textbf{G}}}} \text{ and } \textbf{H}= V\prn{P_1,\ldots,P_{r_{\textbf{H}}}}
 \end{equation}

\noindent If $\textbf{G}$ is defined over a number field $K$ with associated ring of integers $\mathcal{O}_K$, we can find polynomials $P_1,\ldots,P_r\in\mathcal{O}_K\bkt{X_{i,j}}$  as a defining set by clearing denominators. For instance, in the case when $K = \Q$ and $\mathcal{O}_K= \Z$, these are multivariable integer polynomials.\vspace{2mm}

\noindent For a fixed finite set $X=\set{x_1,\ldots, x_t}$ with associated free group $F(X)$ and any group $G,$ the set of homomorphisms from $F(X)$ to $G,$ denoted by $Hom\prn{F\prn{X},G},$ can be identified with $G^t=G_1\times\ldots\times G_t.$ For any point $\prn{g_1,\ldots,g_t}\in G^t,$ we have an associated homomorphism $\varphi_{\prn{g_1,\ldots,g_t}}:F\prn{X}\longrightarrow G$ given by $\varphi_{\prn{g_1,\ldots,g_t}}\prn{x_i}=g_i.$ For any word $w\in F(X),$ we have a function Eval$_w: Hom(F(X), G)\longrightarrow G $ defined by  Eval$_w(\varphi_{\prn{g_1,\ldots,g_t}})(w)=w(g_1,\ldots,g_t).$ For a finitely presented group $\Gamma,$ we fix a finite presentation $\gen{\gamma_1,\ldots, \gamma_{t}\mid r_1,\ldots, r_{t'} },$ where $X=\set{\gamma_1,\ldots, \gamma_{t}}$ generates $\Gamma$ as a monoid and $\set{r_1,\ldots, r_{t'}}$ is a finite set of relations. If $\textbf{G}$ is a complex affine algebraic subgroup of $Gl_n(n,\C),$ the set $Hom(\Gamma,\textbf{G})$ of homomorphisms $\rho: \Gamma\longrightarrow \textbf{G}$ can be identified with an affine subvariety of $\bf G^t.$ Specifically,
\begin{equation}
    Hom(\Gamma,\textbf{G})=\set{\prn{g_1,\ldots, g_{t}}\in \textbf{G}^t\mid r_j\prn{g_1,\ldots, g_{t}}=I_n \text{ for all } j}
\end{equation}

\noindent If $\Gamma$ is finitely generated, $Hom(\Gamma,\textbf{G})$ is an affine algebraic variety by the Hilbert Basis Theorem.

\vspace{2mm}\noindent The set $Hom(\Gamma,\textbf{G})$ also has a topology induced by the analytic topology on $\bf G^t.$ There is a Zariski open
subset of $Hom(\Gamma,\textbf{G})$ that is smooth in the this topology called the smooth locus, and the functions
Eval$_w: Hom(\Gamma, \textbf{G})\longrightarrow \textbf{G}$ are analytic on the smooth locus. For any subset $S\in \Gamma$ and representation $\rho\in Hom(\Gamma, \textbf{G})$, $\overline{\rho(S)}$ will denote the Zariski closure of $\rho(S)$ in $\textbf{G}$.

 \begin{lm}\label{lm11} (\cite[Lemma 5.1]{L})
    Let ${\bf G}\leq GL\prn{n,\C}$ be a $\overline{\Q}$-algebraic group, $L\leq {\bf G}$ be a finitely generated subgroup, and $\textbf{A}\leq {\bf G}$ be a  $\overline{\Q}$-algebraic subgroup. Then, $H=L\cap\textbf{A}$ is closed in the profinite topology. 

 \end{lm}

\begin{proof} We will reproduce the proof because the bound in Corollary \ref{co1} depends on it.
    Given $g\in L-H$, we need a homomorphism $\varphi:L\longrightarrow Q$ such that $|Q|<\infty$ and $\varphi\prn{g}\notin\varphi\prn{H}.$
We first select polynomials $P_1,...,P_{r_{\textbf{G}}},...,P_{r_{\textbf{A}}}\in\C\bkt{X_{i,j}} $ satisfying (\ref{eq:1}). Since $\textbf{G}$ and $\textbf{A}$ are $\overline{\Q}$-defined, we can select $P_j\in\mathcal{O}_{K_0}\bkt{X_{i,j}}$ for some number field $K_0/\Q.$ We fix a finite set $\set{l_1,\ldots, l_{r_L}}$ that generates $L$  as a monoid. In order to distinguish between elements of $L$ as an abstract group and the explicit elements in ${\bf G},$ we set $l=M_l\in {\bf G}$ for each $l\in L.$  In particular, we have a representation given by $\rho_0:L\longrightarrow{\bf G}$ given by $\rho_0(l_{t})=M_{l_{t}}$. We set $K_L$ to be the field generated over $K_0$ by the set of matrix entries $\set{\prn{M_t}_{i,j}}_{t,i,j}$. It is straightforward to see that $K_L$  is independent of the choice of the generating set for $L.$  Since $L$ is finitely generated, the field $K_L$ has finite transcendence degree over $\Q$ and so $K_L$ is isomorphic to a field of the form $K(T)$  where $K/\Q$ is a number field and  $T=\set{T_1,\ldots, T_d}$ is a transcendental basis (see \cite{L}). For each, $M_{l_t},$ we have $(M_{l_t})_{i,j}=F_{i,j,t}(T)\in K_L$. In particular, we can view the $(i,j)$-entry of the matrix $M_{l_t}$ as a rational function in $d$ variables with coefficients in some number field $K$. Taking the ring $R_L$ 
generated over $\mathcal{O}_{K_0}$ by the set $\set{\prn{M_{l_t}}_{i,j}}_{t,i,j}$, $R_L$ is obtained from $\mathcal{O}_{K_0}\bkt{T_1,\ldots, T_d}$ by inverting
a finite number of integers and polynomials. Any ring homomorphism  $R_L\longrightarrow R$ induces a group homomorphism $GL(n,R_L)\longrightarrow GL(n,R)$, and since $L\leq GL(n,R_L)$, we obtain $L\longrightarrow GL(n,R)$ . If $g\in L- H$ then there exists $r_{\mathbf{G}}<j_g\leq r_{\mathbf{A}} $ such that $Q_g=P_{j_{g}}\prn{\prn{M_l}_{1,1},\ldots,\prn{M_l}_{n,n}}\neq 0$. Using Lemma 2.1 in \cite{KM}, we have a ring homomorphism $\psi_R:R_L\longrightarrow R$ with $|R|<\infty$ such that $\psi_R(Q_g)\neq 0$. Setting, $\rho_{R}:GL(n, R_L)\longrightarrow GL(n,R)$ we assert that $\rho_{R}(g)\notin\rho_{R}(H)$. To see this, set $\overline{   M}_{\eta}=\rho_R(\eta)$ for each $\eta\in L$, and note that $\psi_R(P_j((M_{\eta})_{1,1},\ldots,M_{\eta})_{n,n}) )=P_j((\overline{M}_{\eta})_{1,1},\ldots,(\overline{M}_{\eta})_{n,n} )$ . For each $h\in H$, we know that $P_{j_{l}}\prn{(M_h)_{i,j}}=0$ and so $P_j((\overline{M}_{\eta})_{1,1},\ldots,(\overline{M}_{\eta})_{n,n})= 0$ . However, by selection of $\psi_R$, we know that $\psi_R(Q_g)\neq 0$ and so $\rho_{R}(g)\notin\rho_{R}(H)$ . 
\end{proof}

 \noindent Theorem \ref{th1}  and Lemma \ref{lm11} imply Corollary \ref{co1}. 

\begin{proof} 
    Since $H\leq L$ is word quasiconvex, by Theorem \ref{th1} there is a faithful representation $$\rho_H:L\longrightarrow GL\prn{n,\C}$$ such that $\overline {\rho _H(H)} \cap\rho_H(L)=\rho _H(H)$. We can construct the representation in Theorem \ref{th1} so that $\textbf{G}=\overline{\rho_H(L)}$ and $\textbf{A}=\overline {\rho _H(H)}$ are both $\overline{\Q}$-defined. So, by Lemma \ref{lm11}, we can separate $H$ in $L.$ Next, we quantify the separability of $H$ in $L.$ Toward that end, we need to bound the order of the ring $R$ in the proof of Lemma \ref{lm11} in terms of the word length of the element $g.$ Lemma 2.1 from \cite{KM} bounds the size of $R$ in terms of the coefficient size and degree of the polynomial $Q_g.$ It follows from a discussion on pp 412-413 of \cite{KM} that the coefficients and degree can be bounded in terms of the word length of $g,$ and  the coefficients and degrees of the polynomials $P_j.$ Because the $P_j$ are independent of the word $g,$ there exists a constant $N_0$ such that $|R|\leq \norm{g}^{N_0}.$ By construction, the group $Q$ we seek is a subgroup of $GL(n,R).$ Thus, $|Q|\leq |R|^{n^2}\leq\norm{g}^{N_0n^2}.$ Taking $N=N_0n^2$ completes the proof.

\end{proof}
Acknowledgement:  We are thankful to N. Lazarovich and  M. Hagen for the useful discussions about canonical completions during the Workshop on Cube Complexes and Combinatorial Geometry (Montreal 2023) that inspired this work. We thank Dolciani foundation, Simons foundation and PSC CUNY for the support. We thank the referee for helping to improve the exposition.

Olga Kharlampovich, CUNY, Graduate Center and Hunter College;  

Alina Vdovina, CUNY, Graduate Center and City College.
\end{document}